\documentclass{amsart}
\usepackage{amsthm}
\newtheorem{theorem}{Theorem}
\newtheorem{lemma}[theorem]{Lemma}
\newtheorem{proposition}[theorem]{Proposition}
\theoremstyle{definition}
\newtheorem{definition}[theorem]{Definition}
\newtheorem{problem}[theorem]{Problem}
\newtheorem{algorithm}[theorem]{Algorithm}
\theoremstyle{remark}
\newtheorem{remark}[theorem]{Remark}

\usepackage{amssymb}
\usepackage{empheq}
\usepackage{stmaryrd}
\usepackage{enumerate}
\usepackage[shortlabels]{enumitem}
\usepackage{calc}

\newcommand{\norm}[1]{\left\lVert#1\right\rVert}
\usepackage{mathtools}
\DeclarePairedDelimiter{\ceil}{\lceil}{\rceil}
\DeclarePairedDelimiter{\floor}{\lfloor}{\rfloor}
\DeclarePairedDelimiter{\abs}{\lvert}{\rvert}

\newcommand\blfootnote[1]{%
  \begingroup
  \renewcommand\thefootnote{}\footnote{#1}%
  \addtocounter{footnote}{-1}%
  \endgroup
}

\usepackage[left=2.0cm,%
                right=2.0cm,%
                top=2.5cm,%
                bottom=3.5cm,%
                headheight=12pt,%
                a4paper]{geometry}%

\begin{document}

\title[Quantitative analysis of a subgradient-type method for equilibrium problems]{Quantitative analysis of a subgradient-type method for equilibrium problems$^0$}\blfootnote{${}^0$This paper is a condensed version of the Bachelor 
thesis \cite{Pischke} of the first author written under the supervision of the 2nd author.}

\author[Nicholas Pischke and Ulrich Kohlenbach]{Nicholas Pischke and Ulrich Kohlenbach}
\date{\today}
\maketitle
\vspace*{-5mm}
\begin{center}
{\scriptsize Department of Mathematics, Technische Universit\"at Darmstadt,\\
Schlossgartenstra\ss{}e 7, 64289 Darmstadt, Germany, \ \\ 
E-mails: pischkenicholas@gmail.com,kohlenbach@mathematik.tu-darmstadt.de}
\end{center}

\maketitle
\begin{abstract}
We use techniques originating from the subdiscipline of mathematical logic 
called `proof mining' to provide rates of metastability and - under a 
metric regularity assumption - rates of convergence for a subgradient-type algorithm solving the equilibrium problem in convex optimization over fixed-point sets of firmly nonexpansive mappings. The algorithm is due to H. Iiduka and 
I. Yamada who in 2009 gave a noneffective proof of its convergence.
This case study illustrates the applicability of the logic-based abstract 
quantitative analysis of general forms of Fej\'er monotonicity as given 
by the second author in previous papers. 
\end{abstract}

\section{Introduction}
In \cite{KLN2018}, a general logic-based analysis of abstract forms of 
convergence theorems based on general forms of Fej\'er monotonicity is given.
That paper uses methods from the subdiscipline of mathematical logic called 
`proof mining', which aims at the extraction 
of effective bounds from prima facie nonconstructive proofs by logical transformations (see \cite{Koh2008}
for a book treatment and \cite{Koh18} for a recent survey).

Even in most simple cases of ordinary Fej\'er monotonicity on 
the real line and with all the data involved trivially being computable, 
there, in general, are no computable rates of convergence as one can 
show using methods from computability theory (see the discussion in 
\cite{KLN2018} and - in particular - \cite{Neu2015}) which sharpen 
known `arbitrary slow convergence' phenomena discussed in optimization
to noncomputability results. 

Logically speaking, this is because the formulation of the Cauchy-property of a sequence $(x_n)_{n\in\mathbb{N}}$ (say in a metric space $(X,d)$), that is
\[
\forall k\in\mathbb{N}\exists N\in\mathbb{N}\forall n,m\geq N\left(d(x_n,x_m)< 
\frac{1}{k+1}\right)
\]
is of the form $\forall\exists\forall$ which, in general, is of too high logical complexity (and thus not covered by the general logical metatheorems 
used in proof mining to extract bounds from noneffective proofs).

What can be achieved by the aforementioned logical metatheorems, however, 
are effective rates of so-called metastability
which, moreover, are highly uniform. Metastability is based on 
a (noneffectively equivalent but constructively weakened) reformulation of the convergence or Cauchy statements into what is known in logic as Herbrand normal form. In the context of the above example, this reformulation is given by (here $[n;n+g(n)]:=\{ n,n+1,n+2,\ldots,n+g(n)\}$)
\[
\forall k\in\mathbb{N}\forall g\in\mathbb{N}^\mathbb{N}\exists n\in\mathbb{N}\forall i,j\in [n;n+g(n)]\left(d(x_i,x_j)< \frac{1}{k+1}\right).
\]
This statement is of the general form $\forall\exists$ (considering the leading two universal quantifiers as one and disregarding the last universal quantifier as it is bounded) and for statements of the above form, the logical metatheorems of proof mining guarantee the extractability of highly uniform 
effective bounds on `$\exists n\in\mathbb{N}$' (see \cite{Koh2008}). Such bounds are by now well-known in the literature under the name of rates of metastability (after Tao, see e.g. \cite{Tao2008a, Tao2008b}). 

One important consequence of the Fej\'er monotonicity (in 
the very general sense 
of \cite{KLN2018}) of an iterative sequence $(x_n)_{n\in\mathbb{N}}$ is that effective 
rates of convergence {\bf can} be 
established if some general form of regularity, provided quantitatively 
by a so-called 
modulus of regularity, which generalizes many concepts of regularity used 
in optimization, is given (see \cite{KLAN2019}). The existence of 
regularity usually requires to be in a rather special (`tame') 
context where the 
sets in question are e.g. semialgebraic so that tools from the 
model theory of o-minimal structures can be utilized (see e.g. 
\cite{Ioffe2009,Bolteetal} and - for a concrete example - \cite{Borwein17}).

Our paper is intended as a case study to illustrate how the abstract 
approach from \cite{KLN2018,KLAN2019} can be used in a very concrete 
situation to give a perspicuous quantitative analysis of the algorithm 
being considered. The logic-based notions used in these papers now have 
a concrete mathematical meaning so that the whole treatment can be 
given without any reference to logic. We will indicate in the next 
subsection that we expect that many other algorithms can be analyzed in 
a similar way.

In \cite{IY2009}, the following equilibrium problem over the fixed point 
set of a firmly nonexpansive mapping is studied: 
let $f:\mathbb{R}^N\times\mathbb{R}^N\to\mathbb{R}$ be a function such that:
\begin{enumerate}
\item $f(x,x)=0$ for any $x\in\mathbb{R}^N$;
\item $f(\cdot,y)$ is continuous for any $y$;
\item $f(x,\cdot)$ is convex for any $x$.
\end{enumerate}
Now we can state the equilibrium problem for $f$ (a so-called \emph{equilibrium function}) over the fixed point set $Fix(T)$ of $T,$ where $T:\mathbb{R}^N\to\mathbb{R}^N$ is a firmly nonexpansive mapping 
with a nonempty fixed point set.
\begin{problem}[Equilibrium problem of $f$ over $\mathrm{Fix}(T)$]\label{problem} Find a
\[
u\in\mathrm{EP}(\mathrm{Fix}(T),f):=\{u\in\mathrm{Fix}(T)\mid f(u,y)\geq 0\text{ for all }y\in\mathrm{Fix}(T)\}.
\]
\end{problem}
Iiduka and Yamada proposed a subgradient-type iterative algorithm $(x_n)_{n\in\mathbb{N}}$ 
(see also \cite{IS2003b} for a related algorithm also discussed in 
\cite{IY2009}) 
and showed that it converges to a point in $\mathrm{EP}(\mathrm{Fix}(T),f).$
However, there is no quantitative information given in the theorem. In this paper we provide explicit quantitative versions of this result as well as of several intermediate convergence results such as  
\[
\lim_{n\to\infty}f(y_n,x_n)=0\text{ as well as }\lim_{n\to\infty}\norm{x_n-Tx_n}=0.
\]
As mentioned already, even for $N=1$ and $f\equiv 0$ one can 
construct a simple computable $T$ such that $(x_n)_{n\in\mathbb{N}}$ has no computable rate by adapting a counterexample from \cite{Neu2015}. 

Nevertheless, we present a fully effective and highly uniform rate of 
metastability for the algorithm providing a complete finitary account of 
the main result in \cite{IY2009}.

Moreover, in Section \ref{sec:addass} we even give a rate of convergence, modulo an additional metric regularity 
assumption.
\subsection{Analytical preliminaries}
Throughout, we consider $\mathbb{R}^N$ ($N\geq 1$) as the Euclidean space with the usual inner product $\langle\cdot,\cdot\rangle$ and the induced Euclidean norm $\norm{\cdot}$. With $B_r(x)$ and $\overline{B_r(x)}$, we denote the open and closed ball with radius $r>0$ and center $x\in\mathbb{R}^N$ with respect to $\norm{\cdot}$, respectively. \\
Throughout, if not stated otherwise, let $T:\mathbb{R}^N\to\mathbb{R}^N$ be a \emph{firmly nonexpansive mapping}, that is for all $x,y\in\mathbb{R}^N$:
\[
\norm{Tx-Ty}^2\leq\langle x-y,Tx-Ty\rangle.
\]
In particular, $T$ is also \emph{nonexpansive}.

\subsection{The subgradient method of Iiduka and Yamada}
for the equilibrium problem utilizing the subgradient of the equilibrium function $f$:
\begin{algorithm}[Subgradient-type method for Problem \ref{problem}]\label{algorithm}
Choose $\varepsilon_0\geq 0$, $\lambda_0>0$ and $x_0\in\mathbb{R}^N$ arbitrarily and define $\rho_0:=\norm{x_0}$ and set $n=0$. Then, repeat:
\begin{itemize}
\item Given $x_n\in\mathbb{R}^N$ and $\rho_n\geq 0$, choose $\varepsilon_n\geq 0$ and $\lambda_n>0$.
\item Find a point $y_n\in K_n:=\overline{B_{\rho_n+1}(0)}$ such that
\[
f(y_n,x_n)\geq 0\text{ and }\max_{y\in K_n}f(y,x_n)\leq f(y_n,x_n)+\varepsilon_n.
\]
\item Choose $\xi_n\in \partial f(y_n,\cdot)(x_n)$ arbitrarily, define
\[
x_{n+1}:=T(x_n-\lambda_nf(y_n,x_n)\xi_n)\text{ and }\rho_{n+1}:=\max\{\rho_n,\norm{x_{n+1}}\}
\]
and update $n\to n+1$.
\end{itemize}
\end{algorithm}
As discussed in \cite{IY2009}, this algorithm is based on the combination of 
ideas from two well-known algorithms, namely the hybrid steepest descent method of Yamada \cite{Yam2001} and the scheme of Iusem and Sosa from 
\cite{IS2003b}. Note that the approximate maximum point $y_n$ can be computed 
effectively due to 
the error $\varepsilon_n$ whenever the latter is strictly 
positive and $f(\cdot,x_n)$ 
comes with a modulus of uniform continuity on $K_n$ while for $\varepsilon_n=0$ 
there in general would be no computable point $y_n$ (see \cite{Ko91} for a 
discussion of this point in terms of complexity theory).

In \cite{IY2009}, 
the following theorem on the correctness of the algorithm is established:
\begin{theorem}[Iiduka and Yamada, \cite{IY2009}]\label{thm:main}
Let $\mathrm{Fix}(T)\neq\emptyset$. Assume that there is an $M>0$ with $\norm{\xi_n}\leq M$ for all $n\in\mathbb{N}$. Then the sequences $(x_n)_{n\in\mathbb{N}},(y_n)_{n\in\mathbb{N}}$ generated by the algorithm satisfy:
\begin{enumerate}[(a)]
\item For all $u\in\Omega_n:=\{u'\in\mathrm{Fix}(T)\mid f(y_n,u')\leq 0\}$:
\[
\norm{x_{n+1}-u}^2\leq\norm{x_n-u}^2+\lambda_n(M^2\lambda_n-2)(f(y_n,x_n))^2.
\]
In particular, if $\lambda_n\in [0,2/M^2]$:
\[
\norm{x_{n+1}-u}\leq\norm{x_n-u}.
\]
\item If $\Omega:=\bigcap_{n=1}^\infty\Omega_n\neq\emptyset$ and $\lambda_n\in [a,b]\subseteq (0,2/M^2)$ for some $a,b\geq 0$ and all $n\in\mathbb{N}$, then the sequences $(x_n)_{n\in\mathbb{N}},(y_n)_{n\in\mathbb{N}}$ are bounded and
\[
\lim_{n\to\infty}f(y_n,x_n)=0\text{ as well as }\lim_{n\to\infty}\norm{x_n-Tx_n}=0.
\]
\item If $\varepsilon_n\geq 0$ for all $n$ with $\lim_{n\to\infty}\varepsilon_n=0$, in addition to the requirements for (b), then $(x_n)_{n\in\mathbb{N}}$ converges to a point in $\mathrm{EP}(\mathrm{Fix}(T),f)$.
\end{enumerate}
\end{theorem}
In this paper we establish quantitative versions of the claims in this theorem.

\subsection{The range of the results}
First, let us stress that to consider only sets of fixed points $\mathrm{Fix}(T)$ of a firmly nonexpansive mapping $T$ in Problem \ref{problem} is indeed not limiting in the sense that it still 
allows us to consider arbitrary closed and convex sets $C$ in place of $\mathrm{Fix}(T)$: for any such $C$, the metric projection $P_C$ is a firmly nonexpansive mapping (see e.g. 
\cite{BC2017}) with $\mathrm{Fix}(P_C)=C$. See \cite{IY2009} for further considerations on this version of Problem \ref{problem} over $C$.

From that perspective, Problem \ref{problem} can be seen to indeed encompass many general notions and problems from convex optimization as special cases, including in particular the famous 
Nash-equilibrium problem (as treated in \cite{IY2009}) as well as the convex minimization problem, the variational inequality problem and the vector minimization problem, next to others (see \cite{IS2003b}).

Moreover, allowing arbitrary firmly nonexpansive mappings $T$ in place of plain projections $P_C$ can be beneficial in the concrete practical formulation of particular equilibrium problems, as e.g. 
Iiduka and Yamada show in their work \cite{IY2009} for the example of the previously mentioned Nash-equilibrium problem. Here, while dealing with sets $C$ where, on the one hand, $P_C$ may be 
computationally untractable, while, on the other hand, $C$ can be given by (the intersection of) simple closed convex sets $C_i$ whose projections $P_{C_i}$ are tractable, a firmly nonexpansive 
mapping $T$ can be defined using the tractable projections $P_{C_i}$ which is not a projection itself but fulfills $\mathrm{Fix}(T)=C$ and inherits tractability from the $P_{C_i}$.

And further, many practical choices of such sets $C$ from convex optimization already lend themselves to representations as fixed point sets of firmly nonexpansive mappings, a prime example maybe being 
the set of zeros $\mathrm{zer}A$ of a monotone (or accretive) operator $A$. These zero-sets can be expressed as the set of fixed points of the \emph{resolvent} $J_A$ corresponding to $A$ which is, in 
particular, firmly nonexpansive (see \cite{BC2017} for a comprehensive reference on monotone operators).

We expect that various other algorithms for equilibrium problems over 
suitable sets $C$ can be treated by following a similar analysis provided in this paper (using \cite{KLN2018,KLAN2019}). 

\section{A first quantitative analysis}\label{sec:quantana}
A first consequence of Theorem \ref{thm:main} is the following reformulation of (parts of) part (a).
\begin{lemma}\label{lem:normsqdec}
Let $u\in\Omega\neq\emptyset$ and $M>0$ with $M\geq\norm{\xi_n}$ for all $n\in\mathbb{N}$. Suppose $\lambda_n\in [0,2/M^2]$ for all $n\in\mathbb{N}.$ 
Then
\[
\norm{x_{n+1}-u}^2\leq\norm{x_n-u}^2
\]
for all $n\in\mathbb{N}$. Especially, $\norm{x_{n}-u}^2$ converges.
\end{lemma}
As the required sequence is monotone, we can obtain a direct \emph{rate of metastability} for the sequence $\left(\norm{x_n-u}^2\right)_{n\in\mathbb{N}}$ from the 
next lemma which 
follows immediately from \cite{Koh2008}, Proposition 2.27 and Remark 2.29.

\begin{lemma}[Quantitative version of Lemma \ref{lem:normsqdec}]\label{lem:metastabnormsq}
Let $u\in\Omega\neq\emptyset$ with $c_u\geq\norm{x_0-u}^2$ and let $M>0$ with $M\geq\norm{\xi_n}$ for all $n\in\mathbb{N}$. Further, let $\lambda_n\in [0,2/M^2]$ for 
all $n\in\mathbb{N}$. Then for all $k,K\in\mathbb{N}$ and all $g\in\mathbb{N}^\mathbb{N}$:
\[
\exists n\in [K;\Phi'_1(k,g,c_u,K)]\,\forall i,j\in[n;n+g(n)]\;\left(\left\vert\norm{x_i-u}^2-\norm{x_j-u}^2\right\vert <\frac{1}{k+1}\right)
\]
where
\[
\Phi'_1(k,g,c_u,K):=\tilde{g}^{(\ceil*{c_u(k+1)})}(K)\text{ and }\tilde{g}(n):=n+g(n).
\] Here $\tilde{g}^{(n)}(K)$ denotes the $n$-th iteration of $\tilde{g}$ starting from $K.$ For the special case of $K=0$, we simply write $\Phi_1(k,g,c_u):=\Phi'_1(k,g,c_u,0)$.
\end{lemma}
\begin{proof}
The proof given in \cite{Koh2008}, Proposition 2.27 and Remark 2.29, only provides the case for $K=0$. It is, however, immediately apparent from the proof given there that the argument of $\tilde{g}^{(\ceil*{c_u(k+1)})}$ can be chosen to be an arbitrary $K\in\mathbb{N}$. Similarly, it follows from said proof that the resulting $n$ is then of the form $n=\tilde{g}^{(i)}(K)$ for some $i\leq\ceil*{c_u(k+1)}$. Therefore in particular also $n\geq K$ by construction of $\tilde{g}$.
\end{proof}

A lemma used in the proof (in \cite{IY2009}) of Theorem \ref{thm:main}, part (b) and (c), is the following which is a direct corollary of part (a).
\begin{lemma}[Iiduka and Yamada, \cite{IY2009}, p. 257]\label{lem:boundfsqbynormsq}
Let $u\in\Omega\neq\emptyset$ and $M>0$ with $M\geq\norm{\xi_n}$ for all $n\in\mathbb{N}$. Let further $\lambda_n\in [a,b]\subseteq (0,2/M^2)$ for all $n\in\mathbb{N}$. Then
\[
0\leq -a(M^2b-2)(f(y_n,x_n))^2\leq\norm{x_n-u}^2-\norm{x_{n+1}-u}^2
\]
for all $n\in\mathbb{N}$.
\end{lemma}
Using this lemma, we obtain the following quantitative analysis of the convergence of $f(y_n,x_n)$ towards $0.$
\begin{proposition}[Quantitative version of Theorem \ref{thm:main}, part (b), I]
Let $u\in\Omega\neq\emptyset$, $c_u\geq\norm{x_0-u}^2$, and let $M>0$ with $M\geq\norm{\xi_n}$ for all $n\in\mathbb{N}$. Further, let $\lambda_n\in [a,b]\subseteq (0,2/M^2)$ for all $n\in\mathbb{N}$. Then for all $k\in\mathbb{N}$ and all $g\in\mathbb{N}^\mathbb{N}$:
\[ \exists n\leq\Phi_2(k^2+2k,g,c_u)\,\forall i\in[n; n+g(n)] \ 
\left(f(y_i,x_i)<\frac{1}{k+1}\right) \]
where
\[
\Phi_2(k,g,c_u):=\Phi_1\left(\ceil*{\frac{k+1}{\alpha(a,b,M)}}-1,g+1,c_u\right)\text{ and }\alpha(a,b,M):=-a(M^2b-2)
\]
with $\Phi_1$ as in Lemma \ref{lem:metastabnormsq}.
\end{proposition}
\begin{proof}
Let $k\in\mathbb{N}$ and $g\in\mathbb{N}^\mathbb{N}$ be arbitrary. \\ 
By Lemma \ref{lem:metastabnormsq}
\[
\exists n\leq\Phi_1\left(\ceil*{\frac{k+1}{\alpha(a,b,M)}}-1,g+1,c_u\right)
\]
such that for all $i\in [n;n+g(n)]$ (as then $i+1\in [n;n+g(n)+1]$):
\begin{align*}
(f(y_i,x_i))^2&\leq\frac{1}{-a(M^2b-2)}\left(\norm{x_i-u}^2-\norm{x_{i+1}-u}^2
\right) \\
            &<\frac{1}{-a(M^2b-2)}\frac{1}{\left(\ceil*{\frac{k+1}{\alpha(a,b,M)}}- 1\right)+1}\\
            &\leq\frac{1}{\alpha(a,b,M)}\frac{\alpha(a,b,M)}{k+1} =\frac{1}{k+1}
\end{align*}
where the first inequality follows from Lemma \ref{lem:boundfsqbynormsq}. From this the claim is immediate.
\end{proof}
\begin{lemma}[Iiduka and Yamada, \cite{IY2009}, p. 257]\label{lem:boundasym}
Let $u\in\Omega\neq\emptyset$ be arbitrary and $M>0$ with $M\geq\norm{\xi_n}$ for all $n\in\mathbb{N}$. Let further $\lambda_n\in[a,b]\subseteq (0,2/M^2)$ for all $n\in\mathbb{N}$.
\begin{enumerate}[(i)]
\item For all $n\in\mathbb{N}$:
\[
\norm{x_n-x_{n+1}}^2\leq\norm{x_n-u}^2-\norm{x_{n+1}-u}^2+2Mbf(y_n,x_n)\norm{x_n-x_{n+1}}.
\]
In particular, if $L\geq\mathrm{diam}\{x_n\mid n\in\mathbb{N}\}$:
\[
\norm{x_n-x_{n+1}}^2\leq\norm{x_n-u}^2-\norm{x_{n+1}-u}^2+2MbLf(y_n,x_n).
\]
\item For all $n\in\mathbb{N}$:
\[
\norm{x_{n+1}-Tx_{n+1}}\leq\norm{x_n-x_{n+1}}+Mbf(y_n,x_n).
\]
\end{enumerate}
\end{lemma}

\begin{proposition}[Quantitative version of Theorem \ref{thm:main}, part (b), II]\label{prop:quantverbII}
Let $u\in\Omega\neq\emptyset$, $c_u\geq\norm{x_0-u}^2$, $L\geq\mathrm{diam}\{x_n\mid n\in\mathbb{N}\}$ as well as $M>0$ with $M\geq\norm{\xi_n}$ for all $n\in\mathbb{N}$. Further, let $\lambda_n\in[a,b]\subseteq (0,2/M^2)$ for all $n\in\mathbb{N}$. Then, for all $k\in\mathbb{N}$ and all $g\in\mathbb{N}^\mathbb{N}$:
\[ \exists n\leq\Phi_3(k,g,c_u)\,\forall i\in[n;n+g(n)]\ 
\left(\norm{x_{i}-Tx_{i}}<\frac{1}{k+1}\right), \]
where we have
\[
\Phi_3(k,g,c_u):=\Phi_1\left(\ceil*{(\eta(a,b,M,L))^4(k+1)^4}- 1,g',c_u\right)+1
\]
with $g'(n):=g(n+1)+1$ as well as
\[
\eta(a,b,M,L):=\left(\frac{\sqrt{2MbL}}{\sqrt[4]{\alpha(a,b,M)}}+\frac{Mb}{\sqrt{\alpha(a,b,M)}}+1\right)\text{ and }\alpha(a,b,M):=-a(M^2b-2)
\]
with $\Phi_1$ as in Lemma \ref{lem:metastabnormsq}.
\end{proposition}
\begin{proof}
Let $k\in\mathbb{N}$ and $g\in\mathbb{N}^\mathbb{N}$ be arbitrary. As an abbreviation, we write
\[
\Delta_{u,n}:=\norm{x_n-u}^2-\norm{x_{n+1}-u}^2.
\]
Using Lemma \ref{lem:boundfsqbynormsq}, we at first have
\[
0\le f(y_n,x_n)\leq\frac{1}{\sqrt{\alpha(a,b,M)}}\sqrt{\Delta_{u,n}}.\tag{$*_1$}
\]
Further, we have (using Lemma \ref{lem:boundasym}) for any $u\in\Omega$:
\begin{align*}
\norm{x_{n+1}-Tx_{n+1}}&\leq\norm{x_n-x_{n+1}}+Mbf(y_n,x_n)\\
                       &\leq\sqrt{\Delta_{u,n}}+\sqrt{2MbL}\sqrt{f(y_n,x_n)}+Mbf(y_n,x_n)
\end{align*}
and therefore (using ($*_1$)):
\[
\norm{x_{n+1}-Tx_{n+1}}\leq\frac{\sqrt{2MbL}}{\sqrt[4]{\alpha(a,b,M)}}\sqrt[4]{\Delta_{u,n}}+\left(\frac{Mb}{\sqrt{\alpha(a,b,M)}}+1\right)\sqrt{\Delta_{u,n}}\tag{$*_2$}
\]
for all $n\in\mathbb{N}$.

By Lemma \ref{lem:metastabnormsq}, we have that
\[
\exists m\leq\Phi_3(k,g,c_u)-1
\] 
such that (using ($*_2$)) for all $i\in [m;m+g(m+1)]$:
\begin{align*}
\norm{x_{i+1}-Tx_{i+1}}&<\frac{\sqrt{2MbL}}{\sqrt[4]{\alpha(a,b,M)}}\frac{1}{\eta(a,b,M,L)(k+1)}+\left(\frac{Mb}{\sqrt{\alpha(a,b,M)}}+1\right)\frac{1}{(\eta(a,b,M,L))^2(k+1)^2}\\
                       &\leq\left(\frac{\sqrt{2MbL}}{\sqrt[4]{\alpha(a,b,M)}}+\frac{Mb}{\sqrt{\alpha(a,b,M)}}+1\right)\frac{1}{\eta(a,b,M,L)(k+1)}
                 \leq\frac{1}{k+1}.
\end{align*}
If we define $n=m+1$, then $n\leq\Phi_3(k,g,c_u)$ and
\[
i\in [n;n+g(n)]\Rightarrow i-1\in [m;m+g(m+1)]
\]
and thus by the above we have
\[
\norm{x_{i}-Tx_{i}}<\frac{1}{k+1}
\]
for all $i\in [n;n+g(n)]$.
\end{proof}
\begin{remark}
Note that a bound $L$ on the diameter of $(x_n)_{n\in\mathbb{N}}$ as 
used in Proposition \ref{prop:quantverbII} 
as an input can actually be obtained in terms of $c_u$ by setting $L:=2\sqrt{c_u}$ as we have:
\[
\norm{x_n-x_m}\leq\norm{x_n-u}+\norm{x_m-u}\leq\norm{x_0-u}+\norm{x_0-u}\\ \leq 2\sqrt{c_u}. \]
\end{remark}
To obtain a rate of metastability for the sequence $(x_n)_{n\in\mathbb{N}}$, we apply recent results of Kohlenbach, Leu\c{s}tean and Nicolae \cite{KLN2018} on Fej\'er-monotone sequences. Other examples of application of these recent results are especially the derivation of a quantitative version of asymptotic regularity of compositions of two mappings (see \cite{KLAN2017}). We recall the definition of Fej\'er monotonicity.
\begin{definition}
Let $(X,d)$ be a metric space, $F\subseteq X$ nonempty and $(x_n)_{n\in\mathbb{N}}$ be a sequence in $X$. $(x_n)_{n\in\mathbb{N}}$ is called Fej\'er-monotone with respect to $F$, if for all $n\in\mathbb{N}$ and all $p\in F$:
\[
d(x_{n+1},p)\leq d(x_n,p).
\]
\end{definition}
The authors in  \cite{KLN2018} actually introduce a \emph{generalized} form of Fej\'er monotonicity, but for the purpose of this work, the above is enough. However, we pass to the notion of \emph{uniform} Fej\'er monotonicity, as introduced in \cite{KLN2018}, to formulate the (following) quantitative results.

For this, one considers approximations of the approached set $F$ in form of a descending sequence of sets
\[
AF_k\supseteq AF_{k+1}
\]
for $k\in\mathbb{N}$ with
\[
F=\bigcap_{k\in\mathbb{N}}AF_k.
\]
\begin{definition}
$(x_n)_{n\in\mathbb{N}}$ is called \emph{uniformly} Fej\'er monotone with respect to $F$ and $(AF_k)_{k\in\mathbb{N}}$ if for all $r,n,m\in\mathbb{N}$:
\[
\exists k\in\mathbb{N}\,\forall p\in AF_k\,\forall l\leq m\left(d(x_{n+l},p)<d(x_n,p)+\frac{1}{r+1}\right).
\]
Any function $\chi(n,m,r)$ producing such a $k\in\mathbb{N}$ is called a \emph{modulus of $(x_n)_{n\in\mathbb{N}}$ being uniformly Fej\'er monotone}.
\end{definition}
Under the assumption of $(X,d)$ being boundedly compact, the authors of \cite{KLN2018} obtain (in a slightly generalized setting) an explicit effective 
rate of metastability for the sequence $(x_n)_{n\in\mathbb{N}}.$ 
This rate only depends on the particular uniform quantitative reformulations of the assumptions of the setting such as a modulus of uniform Fej\'er monotonicity and some further quantitative information on the space $(X,d)$ and on how the sequence $(x_n)_{n\in\mathbb{N}}$ approaches the set $F$.\\

With quantitative information on the space $(X,d)$ we here mean explicitly a \emph{modulus of total boundedness} (as defined in \cite{KLN2018}) or a `modulus of bounded compactness'. 
This will be discussed in the proof of Theorem \ref{thm:cquant}.

With quantitative information on how the sequence $(x_n)_{n\in\mathbb{N}}$ approaches the set $F$, we mean a bound on $(x_n)_{n\in\mathbb{N}}$ having \emph{approximate $F$-points}. For this, recall the following definition from \cite{KLN2018}:
\begin{definition}\label{def:appfixedpoints}
$(x_n)_{n\in\mathbb{N}}$ has \emph{approximate} $F$-points if $\forall k\in\mathbb{N}\,\exists N\in\mathbb{N}\left(x_N\in AF_k\right)$. A bound $\Phi(k)$ on `$\exists N\in\mathbb{N}$` is called an \emph{approximate $F$-point bound}.
\end{definition}
As a feasibility check on whether these results can be applied and whether the setup of \cite{IY2009} fits into the above framework, note that part (a) of Theorem \ref{thm:main} can be seen (modulo some refinement of the approximations $\Omega_n$) as hinting the uniform Fej\'er monotonicity of $(x_n)_{n\in\mathbb{N}}$ (as the sequence from Algorithm \ref{algorithm}) with respect to the set $\Omega$ 
being taken as $F.$

We at first focus on whether (quantitative versions of) these properties of uniform Fej\'er monotonicity and approximate $F$-/$\Omega$-points can be obtained by suitably modifying the approximations $\Omega_n$.\\

For this, we need to weaken the conditions of $\Omega_n$ to allow $(x_n)_{n\in\mathbb{N}}$ to lie in them further along the approximation. As none of these $x_n$ is expected to be a fixed point of $T$ or to satisfy $f(y_m,x_n)\leq 0$, we weaken these properties to that of \emph{approximate fixed point} and $f(y_m,u)\leq\frac{1}{k+1}$, respectively. Part (b) of Theorem \ref{thm:main} gives, as a feasibility check, that in the long run $f(y_n,x_n)$ is expected to decrease and that the sequence $x_n$ contains better and better approximate fixed points of $T$.\\

Using this motivation, we define
\[
\Omega'_k:=\biggl\{u\in\mathbb{R}^N\mathrel{\Big|}\norm{u-Tu}\leq\frac{1}{k+1}\text{ and }f(y_j,u)\leq\frac{1}{k+1}\text{ for all }j\leq k\biggr\}
\] which plays the role of $AF_k.$
By construction, we naturally have that $(\Omega'_k)_{k\in\mathbb{N}}$ is descending and 
\[ \Omega=\bigcap_{k\in\mathbb{N}}\Omega'_k.
\]
Further, we obtain the following lemma giving a quantitative version of $(x_n)_{n\in\mathbb{N}}$ having approximate $\Omega$-points with respect to $\left(\Omega'_k\right)_{k\in\mathbb{N}}$ (modulo some quantitative reformulations of the parameters of Algorithm \ref{algorithm}).
\begin{lemma}\label{lem:appomegapointbound}
Let $u\in\Omega\neq\emptyset$ with $c_u\geq\norm{x_0-u}^2$ and let $M>0$ with $M\geq\norm{\xi_n}$ for all $n\in\mathbb{N}$. Let $\lambda_n\in [a,b]\subseteq (0,2/M^2)$ for all $n\in\mathbb{N}$ and let $L\geq\mathrm{diam}\{x_n\mid n\in\mathbb{N}\}$. Further, let $\varepsilon_n\geq 0$ for all $n\in\mathbb{N}$ and let $\varepsilon_n\to 0$ ($n\to\infty$) where $\tau$ is a nondecreasing \emph{rate of convergence} for $\varepsilon_n\to 0$ ($n\to\infty$), that is $\tau(k+1)\ge \tau (k)$ and 
\[
\forall k\in\mathbb{N}\, \forall n\geq\tau(k)\left(\varepsilon_n\leq\frac{1}{k+1}\right).
\]
Then 
\[
\forall k\in\mathbb{N}\,\exists n\leq\Phi(k,a,b,M,L,\tau,c_u)\left(
x_n\in\Omega'_k\right),
\]
where
\[
\Phi(k,a,b,M,L,\tau,c_u):=2\ceil*{c_u\cdot(\sigma(a,b,M,L))^416(k+1)^4}+\max\{k,\tau(2k+1)\}+1\] with 
\[
\sigma(a,b,M,L):=\ceil*{\frac{\sqrt{2MbL}}{\sqrt[4]{\alpha(a,b,M)}}+\frac{Mb+1}{\sqrt{\alpha(a,b,M)}}+1}\text{ and }\alpha(a,b,M):=-a(M^2b-2).
\]
\end{lemma}
\begin{proof}
Let $k\in\mathbb{N}$. We again write
\[
\Delta_{u,n}:=\norm{x_n-u}^2-\norm{x_{n+1}-u}^2.
\]
Again by Lemma \ref{lem:boundfsqbynormsq}, we have
\[ 0\le 
f(y_{n+1},x_{n+1})\leq\frac{1}{\sqrt{\alpha(a,b,M)}}\sqrt{\Delta_{u,n+1}}.\tag{$*_1$}
\]
and as in the proof of Proposition \ref{prop:quantverbII}, we obtain
\[
\norm{x_{n+1}-Tx_{n+1}}\leq\frac{\sqrt{2MbL}}{\sqrt[4]{\alpha(a,b,M)}}\sqrt[4]{\Delta_{u,n}}+\left(\frac{Mb}{\sqrt{\alpha(a,b,M)}}+1\right)\sqrt{\Delta_{u,n}}.\tag{$*_2$}
\]
Note, that for $\overline{2}:n\mapsto 2$ we have
\[
\Phi'_1(k,\overline{2},c_u,K)=2\ceil*{c_u(k+1)}+K
\]
for any $K\in\mathbb{N}$ and so 
\[
\Phi(k,a,b,M,L,\tau,c_u)=\Phi'_1((\sigma(a,b,M,L))^416(k+1)^4-1,\overline{2},c_u,\max\{k,\tau(2k+1)\})+1.
\]
Hence by Lemma \ref{lem:metastabnormsq} (applied to $j:=i+1$ and 
$K:=\max\{ k,\tau(2k+1)\}$), we have that \\
$\exists n\in [K;\Phi(k,a,b,M,L,\tau,c_u)-1]\,\forall i\in [n;n+1]$
\[
\left(\Delta_{u,i}=\norm{x_i-u}^2-\norm{x_{i+1}-u}^2<\frac{1}{\sigma^4(a,b,M,L)16(k+1)^4}\right).
\]
By $n+1\geq n\geq\max\{k,\tau(2k+1)\}$ we get $\varepsilon_{n+1}\leq\frac{1}{2(k+1)}$. Therefore, using ($*_1$) and ($*_2$):
\begin{align*}
&\norm{x_{n+1}-Tx_{n+1}}+f(y_{n+1},x_{n+1})+\varepsilon_{n+1}\\
&\qquad<\frac{\sqrt{2MbL}}{\sqrt[4]{\alpha(a,b,M)}}\frac{1}{\sigma(a,b,M,L)2(k+1)}+\left(\frac{Mb+1}{\sqrt{\alpha(a,b,M)}}+1\right)\frac{1}{(\sigma(a,b,M,L))^24(k+1)^2}+\frac{1}{2(k+1)}\\
&\qquad\leq\ceil*{\frac{\sqrt{2MbL}}{\sqrt[4]{\alpha(a,b,M)}}+\frac{Mb+1}{\sqrt{\alpha(a,b,M)}}+1}\frac{1}{\sigma(a,b,M,L)2(k+1)}+\frac{1}{2(k+1)}=\frac{1}{k+1}
\end{align*}
as $\sigma(a,b,M,L)\geq 1$.\\

By the above, we have
\[
\norm{x_{n+1}-Tx_{n+1}}\leq\frac{1}{k+1}\text{ as well as }f(y_{n+1},x_{n+1})+\varepsilon_{n+1}\leq\frac{1}{k+1},
\]
separately as $f(y_{n+1},x_{n+1})\geq 0$ by definition from Algorithm \ref{algorithm}. Also by the definition of Algorithm \ref{algorithm}, we have
\[
\max_{y\in K_{n+1}}f(y,x_{n+1})\leq f(y_{n+1},x_{n+1})+\varepsilon_{n+1}\leq\frac{1}{k+1}.
\]
By definition of the $K_j$, as $K_j\subseteq K_{j+1}$ and $y_j\in K_j$, we have $y_0,\dots,y_{n+1}\in K_{n+1}$.  Therefore, we have especially
\[
f(y_j,x_{n+1})\leq\frac{1}{k+1}
\]
for all $j\leq n+1$ and as $n+1\geq n\geq\max\{k,\tau(2k+1)\}\geq k$ by definition of $n$, we have $x_{n+1}\in\Omega'_k$ as well as $n+1\le \Phi(k,a,b,M,L,\tau,c_u).$
\end{proof}
The next two lemmas now give the quantitative version of the uniform Fej\'er monotonicity of $(x_n)_{n\in\mathbb{N}}$ with respect to $\Omega$ and $\left(\Omega'_k\right)_{k\in\mathbb{N}}$.
\begin{lemma}\label{lem:boundfejermono}
Let $M>0$ with $M\geq\norm{\xi_n}$ for all $n\in\mathbb{N}$ and let $\lambda_n\in[a,b]\subseteq (0,2/M^2)$ for all $n\in\mathbb{N}$. Now, let $n\in\mathbb{N}$ be fixed. For any $k\geq n$ and any $u\in\Omega'_k$:
\[
\norm{x_{n+1}-u}^2\leq\left(\norm{x_n-u}+\frac{1}{k+1}\right)^2+\frac{1}{k+1}f(y_n,x_n)2b(1+M).
\]
In particular, we have
\[
\norm{x_{n+l}-u}\leq\norm{x_n-u}+\frac{l}{k+1}+\frac{\sqrt{2b(1+M)}}{\sqrt{k+1}}\sum_{j=0}^{l-1}\sqrt{f(y_{n+j},x_{n+j})}
\]
for all $u\in\Omega'_k$ and for all $l\in\mathbb{N}$ with $n+l\leq k+1$ (where for $l=0$ the sum is $0$).
\end{lemma}
\begin{proof}
We give a quantitative analysis of the proof of (3.6) in \cite{IY2009}. At first, note that
$\xi_n\in\partial f(y_n,\cdot)(x_n)$ by the definition of Algorithm \ref{algorithm}. Thus, by the definition of the subgradient, we have especially
\[
f(y_n,u)\geq f(y_n,x_n)+\langle u-x_n,\xi_n\rangle
\]
and thus
\[
\langle x_n-u,\xi_n\rangle\geq f(y_n,x_n)-f(y_n,u)
\]
for all $u\in\Omega'_k$.

Therefore, we have:
\begin{align*}
&\norm{x_{n+1}-u}^2&&=&&\norm{T(x_n-\lambda_nf(y_n,x_n)\xi_n)-u}^2\\
&                  &&=   &&\norm{T(x_n-\lambda_nf(y_n,x_n)\xi_n)-Tu+Tu-u}^2\\
&                  &&\leq&&(\norm{T(x_n-\lambda_nf(y_n,x_n)\xi_n)-Tu}+\norm{u-Tu})^2\\
&                  &&\leq&&(\norm{(x_n-\lambda_nf(y_n,x_n)\xi_n)-u}+\norm{u-Tu})^2\\
&                  &&=   &&\norm{(x_n-\lambda_nf(y_n,x_n)\xi_n)-u}^2+2\norm{(x_n-\lambda_nf(y_n,x_n)\xi_n)-u}\norm{u-Tu}+\norm{u-Tu}^2\\
&                  &&\leq&&\norm{x_n-u}^2-2\lambda_nf(y_n,x_n)\langle x_n-u,\xi_n\rangle+\lambda_n^2(f(y_n,x_n))^2\norm{\xi_n}^2+\\
&                  &&    &&2\norm{u-Tu}(\norm{x_n-u}+\lambda_nf(y_n,x_n)\norm{\xi_n})+\norm{u-Tu}^2\\
&                  &&=   &&(\norm{x_n-u}+\norm{u-Tu})^2+\lambda_n^2(f(y_n,x_n))^2\norm{\xi_n}^2-2\lambda_nf(y_n,x_n)\langle x_n-u,\xi_n\rangle+\\
&                  &&    &&2\norm{u-Tu}\lambda_nf(y_n,x_n)\norm{\xi_n}\\
&                  &&\leq&&(\norm{x_n-u}+\norm{u-Tu})^2+\lambda_n^2(f(y_n,x_n))^2\norm{\xi_n}^2-2\lambda_nf(y_n,x_n)(f(y_n,x_n)-f(y_n,u))+\\
&                  &&    &&2\norm{u-Tu}\lambda_nf(y_n,x_n)\norm{\xi_n}\\
&                  &&\leq&&(\norm{x_n-u}+\norm{u-Tu})^2+\lambda_n(f(y_n,x_n))^2(\lambda_nM^2-2)+2\lambda_nf(y_n,x_n)f(y_n,u)+\\
&                  &&    &&2bMf(y_n,x_n)\frac{1}{k+1}\\
&                  &&\leq&&\left(\norm{x_n-u}+\norm{u-Tu}\right)^2+\frac{1}{k+1}f(y_n,x_n)2b(1+M)\\
&                  &&\leq&&\left(\norm{x_n-u}+\frac{1}{k+1}\right)^2+\frac{1}{k+1}f(y_n,x_n)2b(1+M).\\
\end{align*}
From this, it naturally follows that
\[
\norm{x_{n+1}-u}\leq\norm{x_n-u}+\frac{1}{k+1}+\frac{\sqrt{2b(1+M)}}{\sqrt{k+1}}\sqrt{f(y_n,x_n)}.
\]
The claim
\[
\norm{x_{n+l}-u}\leq\norm{x_n-u}+\frac{l}{k+1}+\frac{\sqrt{2b(1+M)}}{\sqrt{k+1}}\sum_{j=0}^{l-1}\sqrt{f(y_{n+j},x_{n+j})}
\]
follows from this by induction on $l\geq 1$ with $n+l\leq k+1$ (the case of $l=0$ is trivial).
\end{proof}
\begin{lemma}\label{lem:fejermonoquant}
Let $M>0$ with $M\geq\norm{\xi_n}$ for all $n\in\mathbb{N}$ and let $\lambda_n\in [a,b]\subseteq (0,2/M^2)$ for all $n\in\mathbb{N}$. Further let $e\geq f(y_n,x_n)$ for all $n\in\mathbb{N}$. Then $(x_n)_{n\in\mathbb{N}}$ is uniformly Fej\'er monotone with modulus $\chi(n,m,r)$, that is for all $r,n,m\in\mathbb{N}$:
\[
\forall u\in\Omega'_k\,\forall l\leq m\left(\norm{x_{n+l}-u}<\norm{x_n-u}+\frac{1}{r+1}\right)
\]
where
\[
k:=\chi(n,m,r):=\max\left\{n+m,\floor*{(r+1)^2m^2\left(1+\sqrt{2be(1+M)}\right)^2}\right\}.
\]
\end{lemma}
\begin{proof}
Fix $r,n,m\in\mathbb{N}$ and assume $m\geq 1$ without loss of generality. Let $k=\chi(n,m,r)$, $u\in\Omega'_k$ and $l\leq m$. By Lemma \ref{lem:boundfejermono}, we have (as $k+1\geq n+m\geq n+l$ by definition)
\[
\norm{x_{n+l}-u}\leq\norm{x_n-u}+\frac{l}{k+1}+\frac{\sqrt{2b(1+M)}}{\sqrt{k+1}}\sum_{j=0}^{l-1}\sqrt{f(y_{n+j},x_{n+j})}\tag{$*$}
\]
and by using $f(y_n,x_n)\leq e$, we obtain
\begin{align*}
\norm{x_{n+l}-u}&\leq\norm{x_n-u}+\frac{l}{k+1}+\frac{\sqrt{2b(1+M)}}{\sqrt{k+1}}l\sqrt{e}\\
                &\leq\norm{x_n-u}+\frac{m}{k+1}+m\frac{\sqrt{2be(1+M)}}{\sqrt{k+1}}
\end{align*}
from ($*$) as $l\leq m$. As
\[
k\geq\floor*{(r+1)^2m^2\left(1+\sqrt{2be(1+M)}\right)^2},
\]
we obtain
\begin{align*}
\norm{x_{n+l}-u}&\leq\norm{x_n-u}+\frac{m}{(r+1)^2m^2\left(1+\sqrt{2be(1+M)}\right)^2}+m\frac{\sqrt{2be(1+M)}}{(r+1)m\left(1+\sqrt{2be(1+M)}\right)}\\
                &\leq\norm{x_n-u}+\frac{1}{r+1}\frac{1}{m\left(1+\sqrt{2be(1+M)}\right)}\left(m+m\sqrt{2be(1+M)}\right)\\
                &\leq\norm{x_n-u}+\frac{1}{r+1}
\end{align*}
from this.
\end{proof}
Applying Theorem 5.1 from \cite{KLN2018}, we now obtain a 
rate of metastability for $(x_n)_{n\in\mathbb{N}}$.
\begin{theorem}[Quantitative version of Theorem \ref{thm:main}, part (c), I]\label{thm:cquant}
Let $e\geq f(y_n,x_n)$ and $M>0$ with $M\geq\norm{\xi_n}$ for all $n\in\mathbb{N}$. Also, let $\lambda_n\in [a,b]\subseteq (0,2/M^2)$ for all $n\in\mathbb{N}$ as well as $L\geq\mathrm{diam}\{x_n\mid n\in\mathbb{N}\}$. Further, let $\varepsilon_n\geq 0$, $\varepsilon_n\to 0$ $(n\to\infty)$ and $\tau$ be a nondecreasing 
rate of convergence for $\varepsilon_n\to 0$ ($n\to\infty$). Let $u\in\Omega\neq\emptyset$ with $c_u\geq\norm{x_0-u}^2$.\\

Then, for all $k\in\mathbb{N}$ and all $g\in\mathbb{N}^\mathbb{N}$:
\[
\exists n\leq\Sigma(k,g)\,\forall i,j\in[n;n+g(n)]\left(\norm{x_i-x_j}\leq\frac{1}{k+1}\right)
\]
for $\Sigma(k,g):=\Sigma_0(P(k),k,g,\chi,\Phi)$ with 
\[
P(k):=\ceil*{(8k+8)\sqrt{N}L}^N
\]
and
\[
\begin{cases}\Sigma_0(0,k,g,\chi,\Phi):=0,\\
\Sigma_0(n+1,k,g,\chi,\Phi):=\Phi(\chi_g^{max}(\Sigma_0(n,k,g,\chi,\Phi),4k+3)),
\end{cases}
\]
where we define
\begin{gather*}
\chi(n,m,r):=\max\left\{n+m,\floor*{(r+1)^2m^2\beta(b,e,M)}\right\},\\
\chi_g(n,k):=\chi(n,g(n),k),\quad\chi_g^{max}(n,k):=\max\{\chi_g(i,k)\mid i\leq n\},\\
\Phi(k):=2\ceil*{c_u\cdot(\sigma(a,b,M,L))^416(k+1)^4}+\max\{k,\tau(2k+1)\}+1,
\end{gather*}
as well as
\[
\begin{cases}
\alpha(a,b,M):=-a(M^2b-2),\\
\beta(b,e,M):=\left(1+\sqrt{2be(1+M)}\right)^2,\\
\sigma(a,b,M,L):=\ceil*{\frac{\sqrt{2MbL}}{\sqrt[4]{\alpha(a,b,M)}}+\frac{Mb+1}{\sqrt{\alpha(a,b,M)}}+1}.
\end{cases}
\]
\end{theorem}
\begin{proof}
The proof is an application of Theorem 5.1 from \cite{KLN2018} with 
$X:=\overline{B_L(x_0)}$ and $F:=\Omega\cap X, \,AF_k:=\Omega'_k\cap X$ 
(and $G:=H:=Id$). 
Here we use that we have
\[
(x_n)_{n\in\mathbb{N}}\subseteq\overline{B_L(x_0)}
\]
as we assume $L\geq\mathrm{diam}\{x_n\mid n\in\mathbb{N}\}$ and we have $\norm{x_n-x_0}\leq\mathrm{diam}\{x_n\mid n\in\mathbb{N}\}$ by definition of the 
diameter. By Example 2.8 of \cite{KLN2018}, the function $\gamma(k):=\ceil*{2(k+1)\sqrt{N}L}^N$ is a modulus of total boundedness of $\overline{B_L(0)}$ and, considering the definition of (II)-moduli of total boundedness from \cite{KLN2018}, it is straightforward to see that these moduli are `translation-invariant' in the case of normed vector spaces, i.e. any (II)-modulus of total boundedness for a set $A\subseteq\mathbb{R}^N$ is also a (II)-modulus of total boundedness for $A+v:=\{a+v\mid a\in A\}$ with $v\in\mathbb{R}^N$. In our situation, we thus have that the particular $\gamma$ is also a (II-)modulus of total boundedness for $\overline{B_L(x_0)}=\overline{B_L(0)}+x_0$.
By Lemma \ref{lem:appomegapointbound}, $\Phi$ is an approximate $F$-point bound and by Lemma \ref{lem:fejermonoquant}, $\chi$ is a modulus for $(x_n)_{n\in\mathbb{N}}$ being 
uniformly Fej\'er monotone w.r.t. $F$ (and $AF_k$). 
Applying Theorem 5.1 with $\gamma,\Phi,\chi$ gives the result.
\end{proof}
\begin{remark}
In the above theorem, we can obtain a bound $e$ on $f(y_n,x_n)$ in terms of $c_u,a,b$ and $M$ by setting 
\[
e:=\sqrt{\frac{c_u}{-a(M^2b-2)}}
\]
as we have (using Lemma \ref{lem:boundfsqbynormsq}):
\begin{align*}
f(y_n,x_n)&\leq\frac{1}{\sqrt{-a(M^2b-2)}}\sqrt{\norm{x_n-u}^2-\norm{x_{n+1}-u}^2}\\
&\leq\frac{1}{\sqrt{-a(M^2b-2)}}\norm{x_n-u}\\
&\leq\frac{\sqrt{c_u}}{\sqrt{-a(M^2b-2)}}.
\end{align*}
\end{remark}
\begin{remark} The complexity of our rate of metastability is mainly given 
by the fact that the function $\Phi\circ \chi_g$ and so, in particular, 
the `counterfunction' $g$ gets iterated in the definition of $\Sigma_0.$
Some iteration of this sort, however, is unavoidable as the counterexample 
given in \cite{Pischke} to the computability of the rate of convergence 
(already for $N=1$ and $f=0$) shows that an extremely special case of 
the algorithm studied computes the limit of a decreasing sequence 
in $[0,1]$ whose rate of metastability necessarily needs this iteration 
process (see the discussion on p.4 of \cite{KLN2018}). In the next section 
we show that a low-complexity rate of full convergence results under an 
additional metric regularity assumption.
\end{remark}
 
\section{Adding further assumptions}\label{sec:addass}
In this section, we investigate two sets of assumptions to strengthen Theorem \ref{thm:cquant}.
\subsection{Uniform closedness}
In \cite{KLN2018}, the authors introduce the notion of \emph{uniform closedness}, an additional assumption on the way the sets $AF_k$ approach the set $F$ (using the notation of the previous general setting of Definition \ref{def:appfixedpoints}).

We recall the corresponding definition.
\begin{definition}
Let $(X,d)$ be a metric space and $F\subseteq X$ be nonempty. Let $AF_k\subseteq X$ be closed with $AF_k\supseteq AF_{k+1}$ and $F=\bigcap_{k\in\mathbb{N}}AF_k$. $F$ is called \emph{uniformly closed} for $(AF_k)_{k\in\mathbb{N}}$ with moduli $\delta_F,\omega_F:\mathbb{N}\to\mathbb{N}$ if
\[
\forall k\in\mathbb{N}\,\forall p,q\in X\left(q\in AF_{\delta_F(k)}\text{ and }d(p,q)\leq\frac{1}{\omega_F(k)+1}\Rightarrow p\in AF_k\right).
\]
\end{definition}
Under the assumption of uniform closedness, the authors obtain Theorem 5.3 in \cite{KLN2018} as a strengthening of Theorem 5.1. In the following, we will observe that, under further quantitative assumptions on the equilibrium function $f$, $\Omega$ is uniformly closed with the previously defined approximations $\left(\Omega'_k\right)_{k\in\mathbb{N}}$ and compute the corresponding moduli of uniform closedness.

\begin{definition}
Let $g:D\subseteq \mathbb{R}^N\to\mathbb{R}$ be uniformly continuous. A \emph{modulus of uniform continuity for $g$ on $D$} is a function $\sigma:\mathbb{N}\to\mathbb{N}$ such that
\[
\forall k\in\mathbb{N}\,\forall x,x'\in D\left(\norm{x-x'}\leq\frac{1}{\sigma(k)+1}\Rightarrow\abs{g(x)-g(x')}\leq\frac{1}{k+1}\right).
\]
\end{definition}
This notion of modulus of uniform continuity differs from the commonly known \emph{modulus of continuity} in (numerical) analysis but is commonly used in computable and constructive analysis as well as in 
proof mining (see e.g. \cite{Bis1967,Koh2008,Wei2000}). 

We then obtain the following result giving corresponding moduli of uniform closedness in terms of moduli of uniform continuity.
\begin{lemma}\label{lem:modunifclsd}
Let $(y_j)_{j\in\mathbb{N}}$ be a sequence in $\mathbb{R}^N$ (defining $\Omega'_k$) and let $\sigma_j$ be moduli of uniform continuity for $f(y_j,\cdot)$ for all $j\in\mathbb{N}$ on some subset $D\subseteq\mathbb{R}^N$. Then
\[
\forall k\in\mathbb{N}\,\forall u,u'\in D\left( u'\in\Omega'_{\delta_\Omega(k)}\text{ and }\norm{u-u'}\leq\frac{1}{\omega_\Omega(k)+1}\Rightarrow u\in\Omega'_k\right)
\]
where
\[
\begin{cases}
\delta_\Omega(k):=2k+1,\\
\sigma^{max}_j(k):=\max\{\sigma_i(k)\mid i\leq j\},\\
\omega_\Omega(k):=\max\{4k+3,\sigma^{max}_{k}(2k+1)\}.
\end{cases}
\]
\end{lemma}
\begin{proof}
Let $k\in\mathbb{N}$ and let $u,u'\in D$ with $u'\in\Omega'_{2k+1}$ and 
\[
\norm{u-u'}\leq\frac{1}{\omega_\Omega(k)+1}.
\]
$u'\in\Omega'_{2k+1}$ is by definition equivalent to
\[
\norm{u'-Tu'}\leq\frac{1}{2(k+1)}\text{ and }\forall j\leq 2k+1\;\left(f(y_j,u')\leq\frac{1}{2(k+1)}\right).
\]
As $T$ is especially nonexpansive, we have $\norm{Tu-Tu'}\leq\norm{u-u'}$. As $\omega_\Omega(k)\geq 4k+3$, we have further
\[
\norm{u-u'}\leq\frac{1}{\omega_\Omega(k)+1}\leq\frac{1}{4(k+1)}
\]
and thus
\begin{align*}
\norm{u-Tu}&\leq\norm{u-u'}+\norm{u'-Tu'}+\norm{Tu-Tu'}\\
&\leq2\norm{u-u'}+\norm{u'-Tu'}\\
&\leq2\frac{1}{4(k+1)}+\frac{1}{2(k+1)} =\frac{1}{k+1}.
\end{align*}
As $\omega_\Omega(k)\geq\sigma^{max}_{k}(2k+1)\geq \sigma_i(2k+1)$ for all $i\leq k$ by assumption, we have
\[
\norm{u-u'}\leq\frac{1}{\omega_\Omega(k)+1}\leq\frac{1}{\sigma_i(2k+1)+1}
\]
for all $i\leq k$ and as $\sigma_i$ is a modulus of uniform continuity for $f(y_i,\cdot)$, we have
\begin{align*}
f(y_i,u)&\leq f(y_i,u')+\abs{f(y_i,u)-f(y_i,u')}\\
&\leq\frac{1}{2(k+1)}+\frac{1}{2(k+1)} =\frac{1}{k+1}
\end{align*}
for all $i\leq k$. Thus, by definition we have $u\in\Omega_k'$.
\end{proof}
Using these moduli, we obtain the following strengthening of Theorem \ref{thm:cquant} in correspondence to Theorem 5.3 instead of Theorem 5.1 (of \cite{KLN2018}).
\begin{theorem}[Quantitative version of Theorem \ref{thm:main}, part (c), II]\label{thm:cquant2}
In addition to the assumptions of Theorem \ref{thm:cquant}, let $\sigma_j$ be a modulus of uniform continuity for $f(y_j,\cdot)$ for any $j\in\mathbb{N}$ on $\overline{B_L(x_0)}$.\\

Then, for all $k\in\mathbb{N}$ and all $g\in\mathbb{N}^\mathbb{N}$:
\[
\exists n\leq\tilde{\Sigma}(k,g)\,\forall i,j\in[n;n+g(n)]\left(\norm{x_i-x_j}\leq\frac{1}{k+1}\text{ and }x_i\in\Omega'_k\right)
\]
for $\tilde{\Sigma}(k,g):=\Sigma_0(P(k_0),k_0,g,\chi_k,\Phi)$ with $P,\chi,\Phi$ and $\Sigma_0$ as in Theorem \ref{thm:cquant} as well as
\[
k_0:=\max\left\{k,\ceil*{\frac{\omega_\Omega(k)-1}{2}}\right\}\text{ and }\chi_k(n,m,r):=\max\{\delta_{\Omega}(k),\chi(n,m,r)\},
\]
where we define
\[
\begin{cases}
\delta_\Omega(k):=2k+1,\\
\sigma^{max}_j(k):=\max\{\sigma_i(k)\mid i\leq j\},\\
\omega_\Omega(k):=\max\{4k+3,\sigma^{max}_{k}(2k+1)\}.
\end{cases}
\]
\end{theorem}
\begin{proof}
Apply Theorem 5.3 of \cite{KLN2018} under the same considerations 
as in the proof of Theorem \ref{thm:cquant}, using Lemma \ref{lem:modunifclsd} with $D:=\overline{B_L(x_0)}.$ The Lemmas \ref{lem:appomegapointbound} and \ref{lem:fejermonoquant} apply as before.
\end{proof}
\begin{remark}\label{rem:OmegaimpEP}
This theorem is a \emph{finitization} of Theorem \ref{thm:main}, part (c) as it (ineffectively, but elementary) implies back the statement of (c): the metastability trivially implies the Cauchy-statement (and thus convergence) of $(x_n)_{n\in\mathbb{N}}$. Further: for $M\in\mathbb{N}$ and $g:n\mapsto M$, Theorem \ref{thm:cquant2} gives $\exists i\geq M\left(x_i\in\Omega_k'\right)$. Thus, as $\Omega_k'$ is closed, we have $x:=\lim_{n\to\infty}x_n\in\Omega_k'$ and as $k$ was arbitrary, we have $x\in\Omega$ by $\Omega=\bigcap_{k\in\mathbb{N}}\Omega'_k$. As in \cite{IY2009}, p. 258, it follows elementary that 
$x\in\mathrm{EP}(\mathrm{Fix}(T),f).$
\end{remark}

\subsection{Regularity conditions}
Using the recent quantitative treatment \cite{KLAN2019} of very general scenarios of regularity conditions in the context of Fej\'er monotone sequences, we can give an improvement of Theorems \ref{thm:cquant} and \ref{thm:cquant2} by adding assumptions on (a quantitative version of) a regularity condition for $\Omega$ and obtain (under this assumption) even rates of convergence for the sequence approximating an equilibrium point.\\

Central for the further results is the following quantitative version of regularity, defined as \emph{modulus of regularity} in \cite{KLAN2019}. For this, given a function $F:\mathbb{R}^N\to\mathbb{R}$, we write $\mathrm{zer}\,F$ for the set of zeros of $F$. 
\begin{definition}[\cite{KLAN2019}]
Let $F:\mathbb{R}^N\to\mathbb{R}$ be a function with $\mathrm{zer}\,F\neq\emptyset$ and fix $z\in\mathrm{zer}\,F$ and $r>0$. A function $\phi:(0,\infty)\to(0,\infty)$ is a \emph{modulus of regularity} for $F$ w.r.t. $\mathrm{zer}\,F$ and $\overline{B_r(z)}$ if for all $\varepsilon>0$ and all $x\in\overline{B_r(z)}$:
\[
\vert F(x)\vert<\phi(\varepsilon)\Rightarrow\mathrm{dist}(x,\mathrm{zer}\,F)<\varepsilon.
\]
\end{definition}
The setting for these regularity conditions in \cite{KLAN2019} is far more general, e.g. being in the context of abstract metric spaces. We will, however, only need the above version for functions over $\mathbb{R}^N$.\\

Under the assumption of a modulus of regularity for $F$ together with the Fej\'er monotonicity of a sequence $(x_n)_{n\in\mathbb{N}}$ w.r.t. $\mathrm{zer}\,F$ and some further assumptions on quantitative information on how the sequence $(x_n)_{n\in\mathbb{N}}$ interacts with $F$, the authors obtain effective rates of convergence for the sequence $(x_n)_{n\in\mathbb{N}}$.

By quantitative information on the interaction of $(x_n)_{n\in\mathbb{N}}$ with $F$, we mean precisely that $(x_n)_{n\in\mathbb{N}}$ has \emph{approximate $F$ zeros}. For this, recall 
the following (modification of the) definition from \cite{KLAN2019}.
\begin{definition}
Let $F$ be as above. We say that a sequence $(x_n)_{n\in\mathbb{N}}$ has \emph{approximate $F$ zeros} if
\[
\forall k\in\mathbb{N}\,\exists n\in\mathbb{N}\left( |F(x_n)|<\frac{1}{k+1}\right).
\]
A bound on `$\exists n\in\mathbb{N}$' is called an \emph{approximate zero bound}.
\end{definition}
Notice the similarity of approximate zeros with approximate $F$-points from Definition \ref{def:appfixedpoints} (although there $F$ has a different meaning). Guided by this similarity, the fact that the particular sequence $(x_n)_{n\in\mathbb{N}}$ from Algorithm \ref{algorithm} has approximate $\Omega$-points (relative to the representation $(\Omega'_k)_{k\in\mathbb{N}}$) and the fact that $(x_n)_{n\in\mathbb{N}}$ is Fej\'er monotone w.r.t. $\Omega$, we are particularly interested in a function $F$ where $(1) \ \mathrm{zer}\,F=\Omega$ and $
(2) \ \abs{F(x)}<1/(k+1)$ relates to $x \in \Omega'_k.$ 
Towards a particular choice, we first define the set-valued function $\gamma:\mathbb{R}^N\to\mathcal{P}(\mathbb{N})$ through
\[
\gamma(x):=\left\{k\in\mathbb{N}\mid f(y_j,x)>\frac{1}{k+1}\text{ for some }j\leq k\right\}
\]
and the corresponding function $G:\mathbb{R}^N\to\mathbb{R}_{\ge 0}$ defined through
\[
G(x):=\begin{cases}0&\text{if }\gamma(x)=\emptyset,\\2&\text{if }\gamma(x)=\mathbb{N},\\\frac{1}{\inf\gamma(x)}&\text{otherwise}.\end{cases}
\]
Given a mapping $T:\mathbb{R}^N\to\mathbb{R}^N$, we may further define the function $F:\mathbb{R}^N\to\mathbb{R}_{\ge 0}$
\[
F(x):=\max\{\norm{x-Tx},G(x)\}.
\]
This function is now an adequate choice which fulfills the previously desired requirements (1) and (2). For this, note first that $\gamma(x)$ has the following property whose proof is immediate:
\begin{lemma}\label{lem:propofgamma}
For any $x\in\mathbb{R}^N$ and any $k\in\mathbb{N}$: $k\in\gamma(x)$ implies $j\in\gamma(x)$ for all $j\geq k$.
\end{lemma}
In the following, we verify properties (1) and (2) for $F$:
\begin{lemma}\label{lem:propofF}
For all $x\in\mathbb{R}^N$ and all $k\in\mathbb{N}$: $x\in\Omega'_k$ iff $F(x)\leq\frac{1}{k+1}$. In particular, $\Omega =\mathrm{zer}\,F$.
\end{lemma}
\begin{proof}
Let $x\in\mathbb{R}^N$ and let $k\in\mathbb{N}$ and suppose first that $x\in\Omega'_k$. Then, we have
\[
\norm{x-Tx}\leq\frac{1}{k+1}\text{ and }f(y_j,x)\leq\frac{1}{k+1}\text{ for all }j\leq k.
\]
The latter gives $k\not\in\gamma(x)$ and thus $j\not\in\gamma(x)$ for all $j\leq k$ by the contraposition of Lemma \ref{lem:propofgamma}. Now, either $\gamma(x)=\emptyset$ or $\gamma(x)\neq\emptyset$. The former gives $G(x)=0\leq\frac{1}{k+1}$ by definition, the latter gives
$\inf\gamma(x)\geq k+1$ 
and thus $G(x)=1/(\inf\gamma(x))\leq 1/(k+1)$. In any way $F(x)\leq1/(k+1)$.

Now, suppose $x\not\in\Omega'_k$. Then, $\norm{x-Tx}>\frac{1}{k+1}$ or $f(y_j,x)>\frac{1}{k+1}$ for some $j\leq k$. The former gives $F(x)>\frac{1}{k+1}$ immediately. The latter gives $k\in\gamma(x)$. Thus,
\[
\inf\gamma(x)\leq k<k+1
\]
and therefore, we have either $\gamma(x)=\mathbb{N}$ where $G(x)=2>\frac{1}{k+1}$ by definition, or $\inf\gamma(x)\geq 1$ where then
\[
G(x)=\frac{1}{\inf\gamma(x)}>\frac{1}{k+1}.
\]
In any way $F(x)>1/(k+1)$.
\end{proof}
Together with the previous results on approximate $\Omega$-points from Lemma \ref{lem:appomegapointbound}, we obtain the following result regarding approximate $F$ zeros.
\begin{lemma}\label{lem:appzerosbound}
Let $u\in\Omega\neq\emptyset$ with $c_u\geq\norm{x_0-u}^2$ and let $M>0$ with $M\geq\norm{\xi_n}$ for all $n\in\mathbb{N}$. Let $\lambda_n\in [a,b]\subseteq (0,2/M^2)$ for all $n\in\mathbb{N}$ as well as $L\geq\mathrm{diam}\{x_n\mid n\in\mathbb{N}\}$. Further, let $\varepsilon_n\geq 0$ for all $n\in\mathbb{N}$ and let $\varepsilon_n\to 0$ ($n\to\infty$) where $\tau$ is a nondecreasing \emph{rate of convergence} for $\varepsilon_n\to 0$ ($n\to\infty$). Let $\Phi$ be as in Lemma 
\ref{lem:appomegapointbound}
Then for any $k\in\mathbb{N}$
\[
\exists n\leq\Phi(k,a,b,M,L,\tau,c_u)\left(F(x_{n})\leq\frac{1}{k+1}\right).
\]
\end{lemma}
\begin{proof}
By Lemma \ref{lem:appomegapointbound}, we have
\[
\exists n\leq\Phi(k,a,b,M,L,\tau,c_u)\;\left(x_{n}\in\Omega'_k\right).
\]
By Lemma \ref{lem:propofF}, this implies $F(x_{n})\leq\frac{1}{k+1}$.
\end{proof}
As the function $F$ may be perceived to be quite artificial, it is of 
interest to see equivalent characterizations for the existence of a modulus of regularity for $F$. The following easy consequence of Lemma 
\ref{lem:propofF} gives a result in this vein.
\begin{lemma}\label{lem:modregchar}
Let $u\in\Omega$ and let $r>0$. Then:
\begin{enumerate}
\item If $\phi:(0,\infty)\to(0,\infty)$ is a modulus of regularity for $F$ w.r.t. $\Omega$ and $\overline{B_r(u)}$, then
\[
\forall\varepsilon>0\,\forall x\in\overline{B_r(u)}\left(x\in\Omega_{\ceil*{\frac{1}{\phi(\varepsilon)}}}'\Rightarrow\mathrm{dist}(x,\Omega)<\varepsilon\right).
\]
\item If $\psi:(0,\infty)\to\mathbb{N}$ is such that
\[
\forall\varepsilon>0\,\forall x\in\overline{B_r(u)}\left(x\in\Omega_{\psi(\varepsilon)}'\Rightarrow\mathrm{dist}(x,\Omega)<\varepsilon\right),
\]
then $1/(\psi(\varepsilon)+1)$ is a modulus of regularity for $F$ w.r.t. $\Omega$ and $\overline{B_r(u)}$.
\end{enumerate}
\end{lemma}

Using this lemma, we obtain the following rate of convergence for the sequence $(x_n)_{n\in\mathbb{N}}$ generated by Algorithm \ref{algorithm} under the assumption of a modulus of regularity for $F$.
\begin{theorem}
Let $u\in\Omega\neq\emptyset$ with $c_u\geq\norm{x_0-u}^2$ and let $M>0$ with $M\geq\norm{\xi_n}$ for all $n\in\mathbb{N}$. Let $\lambda_n\in [a,b]\subseteq (0,2/M^2)$ for all $n\in\mathbb{N}$ as well as $L\geq\mathrm{diam}\{x_n\mid n\in\mathbb{N}\}$. Further, let $\varepsilon_n\geq 0$ for all $n\in\mathbb{N}$ and suppose $\varepsilon_n\to 0$ ($n\to\infty$) with a nondecreasing \emph{rate of convergence} $\tau$.\\

If $\psi:(0,\infty)\to\mathbb{N}$ is such that
\[
\forall\varepsilon>0\,\forall x\in\overline{B_{\sqrt{c_u}}(u)}\left(x\in\Omega_{\psi(\varepsilon)}'\Rightarrow\mathrm{dist}(x,\Omega)<\varepsilon\right),
\]
then $(x_n)_{n\in\mathbb{N}}$ is convergent with $x=\lim_{n\to\infty}x_n\in\mathrm{EP}(\mathrm{Fix}(T),f)$ and
\[
\forall k\in\mathbb{N}\,\forall n\geq\Phi\left(\phi\left(\frac{1}{2(k+1)}\right)\right)\left(\norm{x_n-x}<\frac{1}{k+1}\right)
\]
where we define
\[
\phi(\varepsilon):=\frac{1}{\psi(\varepsilon)+1}
\]
and
\[
\Phi(\varepsilon):=2\ceil*{c_u\cdot(\sigma(a,b,M,L))^416(\ceil*{\frac{1}{\varepsilon}}+1)^4}+\max\left\{\ceil*{\frac{1}{\varepsilon}},\tau\left(2\ceil*{\frac{1}{\varepsilon}}+1\right)\right\}+1,
\]
as well as 
\[
\alpha(a,b,M):=-a(M^2b-2),\
\sigma(a,b,M,L):=\ceil*{\frac{\sqrt{2MbL}}{\sqrt[4]{\alpha(a,b,M)}}+\frac{Mb+1}{\sqrt{\alpha(a,b,M)}}+1}.
\]
\end{theorem}
\begin{proof}
The proof is an application of Theorem 4.1, (i) of \cite{KLAN2019}. Fej\'er monotonicity of $(x_n)_{n\in\mathbb{N}}$ w.r.t $\Omega$ is already contained in (a) of Theorem \ref{thm:main}. By Lemma \ref{lem:modregchar}, (2), $\phi$ is a modulus of regularity for $F$.

Now, let $\varepsilon>0$. Then, by Lemma \ref{lem:appzerosbound}, we obtain
\[
\exists n\leq\Phi(\varepsilon)\;\left( F(x_n)\leq\frac{1}{\ceil*{1/\varepsilon}+1} <\varepsilon\right).
\]
As $\varepsilon$ was arbitrary, Theorem 4.1 of \cite{KLAN2019} applies as $\Omega$ is closed and we obtain $x=\lim_{n\to\infty}x_n\in\Omega$ with the desired rate of convergence. We obtain $x\in\mathrm{EP}(\mathrm{Fix}(T),f)$  as in \cite{IY2009}, p. 258.
\end{proof}
\vspace*{-2mm}
\noindent
{\bf Acknowledgment:} The second author was supported by the German Science 
Foundation (DFG KO 1737/6-1).


\end{document}